\documentclass[12pt]{amsart}
\usepackage{amscd,amssymb,amsthm,verbatim,slashed}
\newcommand{\C}{{\mathbb C}}

\newcommand{\const}{\operatorname{const.}}

\newcommand{\dvol}{\operatorname{dvol}}

\newcommand{\Hess}{\operatorname{Hess}}

\newcommand{\HH}{\operatorname{H}}

\newcommand{\R}{{\mathbb R}}

\newcommand{\Ric}{\operatorname{Ric}}

\newcommand{\Rm}{\operatorname{Rm}}

\newcommand{\Tr}{\operatorname{Tr}}

\newcommand{\vol}{\operatorname{vol}}
\newcommand{\Z}{{\mathbb Z}}

\numberwithin{equation}{section}
\theoremstyle{plain}
\newtheorem{definition}{Definition}

\newtheorem{lemma}{Lemma}
\newtheorem{theorem}{Theorem}
\newtheorem{proposition}{Proposition}
\newtheorem{corollary}{Corollary}

\theoremstyle{remark}

\newtheorem{example}{Example}
\usepackage{graphicx}
\input{amssym.def}
\input{amssym}
\input{epsf}
\usepackage{a4wide}

\begin{document}

\title[Long-time behavior of Ricci flow on some complex surfaces]
      {Long-time behavior of Ricci flow on some complex surfaces}

\author{John Lott}
\address{Department of Mathematics\\
University of California, Berkeley\\
Berkeley, CA  94720-3840\\
USA} \email{lott@berkeley.edu}

\begin{abstract}
 We give biLipschitz models for the Ricci flow on some 4-manifolds (minimal surfaces of general type), exhibiting a combination of expanding and static behavior.
\end{abstract}

\date{May 3, 2026}

\thanks{This material is based upon work supported by the National Science
Foundation under Grant No. DMS-1928930, while the author was in
residence at the Simons Laufer Mathematical Sciences Institute
(formerly MSRI) in Berkeley, California, during the Fall 2024
semester.}

\maketitle

\section{Introduction} \label{sect1}

The long-time behavior of the Ricci flow, when it exists,
gives a way to construct canonical metrics on manifolds. Such canonical metrics include
Einstein metrics and, more generally, Ricci solitons.
For higher genus surfaces, there is a limit $\lim_{t \rightarrow \infty} t^{-1} g(t)$ which is a
metric of constant sectional curvature in the conformal class of $g(0)$. For compact $3$-manifolds,
there is a conjectural picture for the long-time behavior \cite{Lott (2010)}, although there are still
open questions.

In this paper we look at a four dimensional Ricci flow in which the long-time model is a hybrid of
expanding behavior and static behavior.  We state the results and motivate them afterward.

If $g_1$ and $g_2$ are two Riemannian metrics on a manifold, and $K \ge 1$, then we say that the metrics
are $K$-biLipschitz (relative to the identity map) if
$K^{-1} g_1 \le g_2 \le K g_1$. Explicit model flows $g_{mod}(t)$ and
$g^{(k)}_{mod}(t)$ will be described later.

We will identify a K\"ahler metric $g$ with its associated K\"ahler form $\omega$.

\begin{theorem} \label{thm1}
Let $M$ be a minimal complex surface of general type
whose canonical model $X$ has only isolated $\Z_2$-quotient singularities.
Equivalently, the exceptional locus of $M \rightarrow X$ is a disjoint union 
of rational curves $\{E_i\}$
with self intersection $-2$.
Let $[E_i] \in \HH^{1,1}(M; \R)$ be the cohomology class dual to the homology class of $E_i$.
Given positive numbers $\{b_i\}$, there is a model flow of K\"ahler metrics
$g_{mod}(t)$ on $M$, defined for large $t$, so that 
\begin{itemize}
\item As cohomology classes, $[\omega_{mod}(t)] = - \pi \sum_i b_i^{-1} [E_i] - 2 \pi t c_1(M)$, and
\item For all sufficiently large $T$, the Ricci flow solution $g(t)$ defined for $t \ge T$ with 
initial condition $g(T) = g_{mod}(T)$ is such
that for all $\epsilon > 0$, $g(t)$ is $K(t)$-biLipschitz to $g_{mod}(t)$, where $K(t) = 1 + O \left( 
t^{- \: \frac23 +\epsilon} \right)$.
\end{itemize}
\end{theorem}

\begin{corollary} \label{cor1}
In the setting of Theorem \ref{thm1}, if $m \in E_i$ then $\lim_{s \rightarrow \infty} (M, g(s+\cdot), m)$
exists in the pointed
Lipschitz Cheeger-Hamilton topology. The limit is the static flow of $b_i^{-1}$ times the Eguchi-Hanson metric.
\end{corollary}

We note that to get a limit in Corollary \ref{cor1}, one must perform $s$-dependent diffeomorphisms.

The canonical model $X$ is an orbifold that admits a
K\"ahler-Einstein metric of negative Einstein constant.  If $X$ is complex hyperbolic then
there is an improved convergence statement.

\begin{theorem} \label{thm2}
Suppose that $X$ is a complex hyperbolic orbifold. 
For every $\epsilon > 0$, there is a model flow $g_{mod}^{(k)}(t)$ so that the result of Theorem \ref{thm1} holds with
$K(t) = 1 + O \left( 
t^{- 2 + \epsilon} \right)$.
\end{theorem}

There is also a stability result in this case (Proposition \ref{prop1}).

To motivate the theorems, we begin with a general discussion of K\"ahler-Ricci flow. Let $(M, g(t))$ be a K\"ahler-Ricci flow on a compact complex manifold $M$ of arbitrary dimension, with initial metric
$g(0)$.  Tian and Zhang showed that
the flow is immortal, i.e. exists for all positive time, if and only if the canonical bundle
$K_M$ is nef, i.e. $c_1(K_M)$ is in the closure of the K\"ahler cone of $M$
\cite{Tian-Zhang (2006)}. 

We restrict to immortal flows.
Assuming the Abundance Conjecture (which is known for K\"ahler surfaces),
Song and Tian showed that the scalar curvature $R$ is $O(t^{-1})$ in magnitude as $t \rightarrow \infty$
\cite{Song-Tian (2016)}. 
Immortal Ricci flows are divided into types III and II, depending on whether or not the sectional
curvatures decay in magnitude like $O(t^{-1})$.
If the flow is type-II, i.e. if $\|\Rm(g(t))\|_\infty$ fails to be $O(t^{-1})$,
then we can take a type-II rescaling limit.  That is,
we can find a sequence of spacetime points $\{(m_i, t_i)\}_{i=1}^\infty$ so that
$|\Rm|$ achieves its maximal time-$t_i$ value $Q_i$ at $m_i$, with $\lim_{i \rightarrow \infty}
t_i Q_i = \infty$, and so that there is a pointed
smooth Cheeger-Hamilton limit $\lim_{i \rightarrow \infty} (M_i, Q_i g(t_i + Q_i^{-1} s), m_i) = (M_\infty, g_\infty(s),
m_\infty)$ \cite[Chapter 8.2.1.3]{Chow-Lu-Ni (2006)}. Here the limit is an eternal Ricci flow,
i.e. exists for all $s \in \R$, and is not flat.  
It lives on a manifold if one has the relevant injectivity radius
lower bounds at $\{m_i\}_{i=1}^\infty$, and otherwise lives on an \'etale groupoid \cite{Lott (2007)}.  
Because of the rescaling, the limiting
scalar curvature $R(g_\infty)$ vanishes.  Then from the evolution equation for scalar curvature
$\frac{\partial R}{\partial t} = \triangle R + 2 |\Ric|^2$, one concludes that
$g_\infty(s)$ is Ricci flat and hence constant in $s$. In short, if an immortal K\"ahler-Ricci flow is
type-II then one expects to extract a nontrivial Ricci flat space in a scaling limit. 

In this paper we look at immortal K\"ahler-Ricci flows on compact complex manifolds $M$ of complex dimension two. A survey is in \cite{Tosatti (2024)}. For cohomological reasons, the volume growth
is a polynomial in the time $t$. If it is constant in $t$ then $M$ is a Calabi-Yau manifold and
Cao showed that
the K\"ahler-Ricci flow approaches the Ricci flat metric in the given K\"ahler class
\cite{Cao (1985)}. If the volume is linear in $t$ then $M$ is an elliptic surface and the
K\"ahler-Ricci flow was studied by Song and Tian \cite{Song-Tian (2007)}. We are concerned with the
remaining case, when the volume is quadratic in $t$, which was studied by
Tian and Zhang \cite{Tian-Zhang (2006)}. In this case the canonical bundle $K_M$ is big
in the sense that $c_1(K_M)^2 \neq 0$. 

Equivalently, one could say that we are looking at the 
K\"ahler-Ricci flow on projective surfaces of general type.  The K\"ahler-Ricci flow on such a 
surface may encounter singularities corresponding to undoing blowups, i.e. contracting rational curves
of self intersection $-1$ to points.  It is known how to flow through such singularities
\cite{Song-Tian (2017),Song-Weinkove (2013)}, of which there is a finite number,
so we are reduced to studying K\"ahler-Ricci flows
on minimal projective surfaces of general type.

Such a manifold $M$ has a canonical model $X$, which is an orbifold of complex dimension two
with isolated singularities.
There is a morphism $p : M \rightarrow X$, a resolution of singularities, 
so that if $x \in X$ is a regular point then
$p^{-1}(x)$ is a point, while if $x$ is a singular point then $p^{-1}(x)$ is a connected union of rational
curves $E$ of self intersection $-2$
\cite[Chapter VII.5]{Barth-Hulek-Peters-VandeVen (2004)}.
(This is more general than the setting of Theorem \ref{thm1}.)
There is a unique K\"ahler-Einstein metric $g_{KE}$ on $X$ with 
$\Ric_{KE} = - \omega_{KE}$ \cite{Kobayashi (1985)}.

If $X$ is smooth, i.e. if $M$ is already a canonical model, then the K\"ahler-Ricci flow on $M$,
starting from any initial K\"ahler metric $g(0)$, has the property that
$\lim_{t \rightarrow \infty} t^{-1} g(t) = g_{KE}$ smoothly
\cite{Cao (1985)}. On the other hand, if $X$ is not smooth then the flow on $M$ is type-II
\cite{Tosatti-Zhang (2015)}.
Tian and Zhang showed that $\lim_{t \rightarrow \infty} t^{-1} \omega(t) = \omega_{KE}$ as
a current, with smooth convergence on compact subsets of $M - \bigcup_i E_i$, where we identify the latter with the regular part of $X$
\cite{Tian-Zhang (2006)}. In particular, the geometry away from the $-2$ rational curves $\{E_i\}$
of $M$ is
asymptotically linearly expanding in $t$.
However, the asymptotic behavior of the K\"ahler-Ricci flow on all of $M$ is less clear. 

From this point on we restrict to the case where the exceptional curves $\{E_i\}$ are disjoint, i.e. when the singular points of $X$ all have isotropy group $\Z_2$.
Then $M$ can be reconstructed from $X$ as follows.  If $x$ is an orbifold point of $X$
then there is a neighborhood $U_x$ that is analytically equivalent to $B(0, \delta)/\Z_2$, 
the $\Z_2$-quotient of a ball in $\C^2$. On the other hand, there is a morphism $q : T^* \C P^1
\rightarrow \C^2/\Z_2$, a resolution of the singularity at the vertex of the cone.
It can be seen as the $\Z_2$-quotient of the blowdown $O(-1) \rightarrow \C^2$, where we
identify the blowup of $\C^2$ at the origin with the 
$O(-1)$ line bundle on $\C P^1$, and
identify $T^* \C P^1$ with $O(-2)$.
We can remove $U_x$ from $X$ and
glue in $q^{-1}(B(0, \delta)/\Z_2)$. Then $M$ is the result of doing such an operation for
each singular point of $X$.

Under the K\"ahler-Ricci flow, 
the area of each curve $E_i$ is constant in time, since the adjunction formula implies that
$\int_{E_i} c_1(K_M) = 0$.
This is in contrast to the expanding behavior
away from $\bigcup_i E_i$. Hence it is not so clear what the global model should be.
To put it another way, there is a Ricci flat K\"ahler metric on $T^* \C P^1$, the Eguchi-Hanson
metric.  At spatial infinity, it is asymptotic to $\C^2/\Z_2$. One can construct a K\"ahler metric on $M$ by taking a large piece of the
Eguchi-Hanson metric, scaling it down and gluing it onto $(X, g_{KE})$ to replace the singularities.
While one can do this at a given time, if one lets it evolve under the K\"ahler-Ricci flow then the Eguchi-Hanson region wants to
remain static, while $M$ wants to expand outside of $\bigcup_i E_i$. It isn't immediately clear how the
evolution mediates between these two conflicting tendencies.

The solution to this problem comes from the fact
that there is actually a two parameter family of Eguchi-Hanson
metrics on $T^* \C P^1$. One parameter just comes from multiplicative rescaling.  The other parameter comes from pulling back an Eguchi-Hanson metric by automorphisms of $T^* \C P^1$ 
that act by rescaling the cotangent fiber.  In terms of the morphism $q : T^* \C P^1 \rightarrow
\C^2/\Z_2$, these automorphisms
fix the exceptional $\C P^1$ and
push down to a rescaling on $\C^2/\Z_2$. While the pullback gives isometric metrics, they are different
metrics on a fixed $T^* \C P^1$.  Pulling back the Eguchi-Hanson metric 
by a rescaling $z \rightarrow \sqrt{bt} z$ on $\C^2/\Z_2$ and multiplying by $b^{-1}$,
we obtain a $1$-parameter family $g_{EH}^{(0)}(t)$ of K\"ahler metrics on
$T^* \C P^1$
that are asymptotic to
$t g_{flat}$ at spatial infinity and for which the area of the exceptional $\C P^1$ is proportionate to
$b^{-1}$, independent of $t$.  In effect, we are making an artificially
expanding family. This family can be glued to the expanding 
K\"ahler-Ricci flow solution $t g_{KE}$ on $X$, to obtain a $0^{th}$-order approximation to a
K\"ahler-Ricci flow on $M$.

The K\"ahler potential for $g_{EH}^{(0)}(t)$ is given in (\ref{3.1}).  Of course,
$g_{EH}^{(0)}(t)$ is not a Ricci flow solution; the time slices are Ricci flat but the
solution is time dependent.  One finds that the norm of the deviation from solving the Ricci flow
equation decays in time, but unfortunately it does not decay fast enough.  To this end,
we iteratively find a sequence $\{ g_{EH}^{(k)} \}_{k=1}^\infty$ of corrections to $g_{EH}^{(0)}(t)$
that are closer and closer to being K\"ahler-Ricci flow solutions.  It turns out that
$g_{EH}^{(1)}$ is good enough for Theorem \ref{thm1}; its K\"ahler potential is given in (\ref{3.18}).

To construct the model flow $g_{mod}(t)$ for Theorem \ref{thm1}, 
we glue the approximate K\"ahler-Ricci flow 
$g_{EH}^{(1)}(t)$ on $T^* \C P^1$ to the exact K\"ahler-Ricci flow $g_X(t) = t g_{KE}$
on $X$. The curvature of
$g_{EH}^{(1)}(t)$ is concentrated in the region $|z| \le \const t^{- \frac12}$. 
On the other hand,
we do the gluing
at a scale $|z| \sim t^{-a}$, with $a \in (0, \frac12)$, so that the gluing is done in the distant
conical region of $(T^* \C P^1, g_{EH}^{(1)}(t))$. We perform the gluing at the level of K\"ahler
potentials. Both $g_{EH}^{(1)}(t)$ and $g_X(t)$ are approximately conical in the gluing region but
differ in lower orders, which makes the gluing delicate.
It turns out that the best choice for $a$ is $\frac13$. To prove Theorem \ref{thm1} we use a fixed point theorem as in the paper \cite{Brendle-Kapouleas (2017)} of Brendle and Kapouleas.
Since we are in the K\"ahler setting, we can reduce the K\"ahler-Ricci flow to an
evolution equation for the K\"ahler potential, as is customary.  This means that we are dealing with a scalar
equation, which makes the analysis simpler than in \cite{Brendle-Kapouleas (2017)}.

To prove Theorem \ref{thm2}, we use the fact that as $k$ increases, the 
approximate K\"ahler-Ricci flow solutions $g_{EH}^{(k)}(t)$ become better and
better approximations, at large scale, to the evolution of a complex hyperbolic metric.
We can do the gluing at a scale $|z| \sim t^{-a}$ with $a$ arbitrarily small.  Then the
proof of Theorem \ref{thm2} is similar to that of Theorem \ref{thm1},
with the freedom in the choice of $a$ giving the improved convergence.

We mention some earlier related work. On the static side, one can form a
Ricci flat metric on a K3 manifold, using the Kummer construction,
as in the paper by Donaldson \cite{Donaldson (2012)}. One glues Eguchi-Hanson metrics to the 16
singular points in a $\Z_2$-quotient of $T^4$. In the nonK\"ahler case, Brendle and
Kapouleas constructed an ancient Ricci flow on the result of performing the Kummer construction
except reversing the orientation of half of the 16 Eguchi-Hanson spaces
\cite{Brendle-Kapouleas (2017)}. Our treatment of the analytic aspects is taken from 
\cite{Brendle-Kapouleas (2017)}.

While this paper was being written, Deruelle and Ozuch posted a preprint in which they construct
ancient and immortal Ricci flow solutions in the four dimensional nonK\"ahler case via gluing  \cite{Deruelle-Ozuch (2024)}.
They consider oriented Riemannian orbifolds that have isolated singularities 
with certain isotropy groups, such as finite subgroups of $SU(2)$,
and which satisfy a stability condition
at the singular points.  They glue in rescaled regions of Ricci flat ALE manifolds
to construct an ancient or immortal Ricci flow solution, also following
the analytic approach of 
\cite{Brendle-Kapouleas (2017)}, and
get information about the curvature as time goes to $\pm \infty$.
(We get information about the biLipschitz behavior of the metric because we start with
the K\"ahler potential and get estimates about its second derivatives, i.e. the metric, while
Deruelle and Ozuch start with the metric and get estimates about its second derivatives, i.e.
the curvature.)  They give an example of an orbifold that satisfies their stability condition
by
reversing the orientation of a complex hyperbolic orbifold. The immortal solution
obtained by gluing in the ALE spaces is nonK\"ahler.
The minimal $2$-spheres in the ALE regions have areas that increase like $t^{\frac23}$,
whereas in the K\"ahler case the areas are constant in $t$.

Some further questions are:
\begin{enumerate}
\item If one starts with any initial K\"ahler metric on $M$ in the time-$T$ K\"ahler class
of Theorem \ref{thm1}, does the
K\"ahler-Ricci flow approach the model flow?
\item Are there analogous results for all initial K\"ahler classes on $M$?  One would have to
construct model flows using the K\"ahler-Ricci flow on $X$, rather than just uniformly expanding
flows.
\item Can one extend the methods to when the K\"ahler-Einstein orbifold $X$ has isolated singularities
of arbitrary isotropy group?  One would glue in more general Ricci flat ALE K\"ahler manifolds
\cite{Kronheimer (1989)},
rather than just Eguchi-Hanson spaces.
\item Can one extend the biLipschitz closeness in Theorems \ref{thm1} and \ref{thm2} to $C^r$-closeness?
\item Can the methods be extended to elliptic fibrations? Some information about the flow is in
\cite{Song-Tian (2012)}.
\item Are there analogous results in higher dimension or for finite time singularities?
\end{enumerate}

The structure of the paper is as follows.  Section \ref{sect2} has some background information. 
In Section \ref{sect3} we construct approximate
K\"ahler-Ricci flows on $T^* \C P^1$ that are linearly expanding in time at large distances.
We then construct the model flow on $M$ in Section \ref{sect4} by gluing the approximate K\"ahler-Ricci flow on
$T^* \C P^1$ to the expanding K\"ahler-Ricci flow on the orbifold $X$. Theorem \ref{thm1} is proved in
Section \ref{sect5}, and Theorem \ref{thm2} is proved in Section \ref{sect6}.  More detailed descriptions appear at
the beginnings of the sections.

I thank Alan Reid and Song Sun for discussions, and the referee for helpful comments.

\section{Background} \label{sect2}
In this section we review some facts about the K\"ahler-Ricci flow and the Eguchi-Hanson metric.
We will use the Einstein summation convention freely.

\subsection{K\"ahler-Ricci flow and potential flow}
Given a K\"ahler manifold $M$ of complex dimension $n$, the K\"ahler form is a real $(1, 1)$-form 
$\omega$ which can be expressed in holomorphic normal coordinates at a point $p$ by
$\omega(p) = \sqrt{-1}
\sum_{i=1}^n dz^i \wedge d\overline{z}^i$.
Writing $\omega = \sum_{i,j = 1}^n
\sqrt{-1} g_{i \overline{j}} dz^i \wedge
d \overline{z}^j$ locally, the Ricci form is
\begin{equation} \label{2.1}
\Ric = - \sqrt{-1} 
\partial \overline{\partial} \log \frac{\omega^n}{\Omega} =
- \sqrt{-1} \partial \overline{\partial} \log \det \left( g_{i \overline{j}}
\right),
\end{equation}
where $\Omega = n! (\sqrt{-1})^n dz^1 \wedge d\overline{z}^1 \wedge\ldots \wedge dz^n
\wedge d\overline{z}^n$.

The K\"ahler-Ricci flow equation is 
\begin{equation} \label{2.2}
\frac{d\omega}{dt} = - \Ric(\omega).
\end{equation}
The corresponding
cohomology class satisfies $[\omega(t)] = [\omega(0)] - 2 \pi t c_1(M) \in \HH^{1,1}(M; \R)$. 
In local coordinates, if we solve the potential flow equation
\begin{equation} \label{2.3}  
\frac{\partial u}{\partial t} = \log \det \left(  \partial_i \overline{\partial}_j u
\right)
\end{equation}
then $\omega(t) = \sqrt{-1} \partial \overline{\partial} u$ is a solution of
(\ref{2.2}), provided that $\omega(t)$ is a positive $(1,1)$-form.

More globally,
suppose that
$\omega_{mod}(t)$ is a $1$-parameter family of K\"ahler forms so that
$[\omega_{mod}(t)] = [\omega_{mod}(0)] - 2 \pi t c_1(M)$. 
By the $\partial \overline{\partial}$-lemma, if $M$ is compact then we can solve
\begin{equation} \label{2.4}
\frac{d \omega_{mod}}{dt} = - \Ric(\omega_{mod}) + \sqrt{-1} \partial \overline{\partial} f
\end{equation}
for some smooth $1$-parameter
family of real-valued functions $f(t)$.  If $u$ is a solution to the potential flow equation
\begin{equation} \label{2.5}
\frac{\partial u}{\partial t} = \log
\frac{(\omega_{mod} + \sqrt{-1} \partial \overline{\partial} u)^n}{\omega_{mod}^n} 
- f
\end{equation}
then $\omega(t) = \omega_{mod}(t) + \sqrt{-1} \partial \overline{\partial} u(t)$ is a solution to (\ref{2.2}),
provided that it is a positive $(1,1)$-form.  Conversely, any solution of (\ref{2.2}) arises in this way from a
solution to (\ref{2.5}), up to changing $u$ by a function that only depends on $t$.

If $\omega_{KE}$ is the K\"ahler form of a metric with $\Ric(\omega_{KE}) = - \omega_{KE}$ then there is a
K\"ahler-Ricci flow solution $\omega(t) = t \omega_{KE}$.
As a special case, the K\"ahler potential for the complex hyperbolic metric on $B \left( 0, \sqrt{3} \right) \subset \C^2$, normalized so that $\Ric(\omega) = - \omega$, is $- 3 \log
\left( 1 - \frac{1}{3} |z|^2 \right)$. A potential for
the corresponding flow, solving (\ref{2.3}), is $u(t) = 2(t \log t - t)
- 3 t \log
\left( 1 - \frac{1}{3} |z|^2 \right)$.

\subsection{Eguchi-Hanson metric}

A reference is \cite{Lye (2022)}.
The Eguchi-Hanson metric is a Ricci flat K\"ahler metric on $T^* \C P^1$, i.e. on the total space of the
$O(-2)$-bundle on $\C P^1$. The manifold is a resolution of the cone $\C^2/\Z_2$, i.e. there is an analytic map $T^* \C P^1 \rightarrow \C^2/\Z_2$ that is a biholomorphism from the complement of the zero section in $T^* \C P^1$ to the
complement of the vertex $\star$ in $\C^2/\Z_2$.

Fixing the area of the
exceptional $\C P^1$,
there is actually a $1$-parameter family of Eguchi-Hanson metrics on $T^* \C P^1$ that differ by pullback under
fiberwise rescalings of $T^* \C P^1$.
Equivalently, the action is by rescalings of $\C^2/\Z_2 - \{\star\}$, and the identity on $\C P^1 \subset T^* \C P^1$.
Of course, the elements of the $1$-parameter family are mutually isometric, but they describe different metrics on 
the fixed manifold
$T^* \C P^1$.
We will normalize the Eguchi-Hanson metrics as follows.
Restricting 
to $\C^2/\Z_2 - \{\star\}$, we can write a K\"ahler potential for the Eguchi-Hanson metric
as a $\Z_2$-invariant function on 
$\C^2 - (0,0)$, with coordinates $\{z^1, z^2\}$.  Putting
$\rho = |z|^2$, the potentials are given by
\begin{equation} \label{2.6}
\phi_{EH, c} =  \sqrt{1 + c^2 \rho^2} + \frac{1}{2}
\log 
\frac{
\sqrt{1 + c^2 \rho^2} - 1}{
\sqrt{1 + c^2 \rho^2} + 1}
\end{equation} 
for $c > 0$. 
Their derivatives are
$\frac{d}{d\rho} \phi_{EH, c} = \frac{1}{\rho} \sqrt{1 + c^2 \rho^2}$.

Our normalization is such that as $\rho \rightarrow 
\infty$, the potential $\phi_{EH,c}$ is asymptotic to $c\rho = c|z|^2$. That is, the potential approaches $c$ times that of a flat Euclidean cone.

The Eguchi-Hanson metric is
\begin{align} \label{2.7}
g^{EH}_{i \overline{j}} & = \delta_{ij} \frac{\sqrt{1 + c^2 \rho^2}}{\rho} - 
\frac{z^{\overline{i}} z^j}{\rho^2 \sqrt{1 + c^2 \rho^2}} \\
& = \frac{\sqrt{1 + c^2 \rho^2}}{\rho}  \left( \delta_{ij} - \frac{z^{\overline{i}} z^j}{\rho} \right)  +
\frac{c^2 \rho}{\sqrt{1 + c^2 \rho^2}}
\frac{z^{\overline{i}} z^j}{\rho}
, \notag
\end{align}
where the latter expression gives an orthogonal decomposition of $g^{EH}$ in terms of $(\C \vec{z})^\perp$
and $\C \vec{z}$. The inverse metric is
\begin{equation} \label{2.8}
g_{EH}^{\overline{j} i}  = 
\frac{\rho}{\sqrt{1 + c^2 \rho^2}} \left( \delta^{ij} - \frac{z_{\overline{i}} z_j}{\rho} \right)  +
\frac{ \sqrt{1 + c^2 \rho^2}}{c^2 \rho} \frac{z_{\overline{i}} z_j}{\rho}
\end{equation}

\section{Approximate K\"ahler-Ricci flow on the caps} \label{sect3}

In this section we describe potentials for an approximate solution to the potential flow on
$T^* \C P^1$, with the property that it is linearly expanding in time at spatial infinity.
There is a sequence of such approximate solutions that are closer and closer to being
solutions to the potential flow.
We estimate the deviation from being a solution.

We will want to glue the Eguchi-Hanson metric, on a neighborhood of $\C P^1 \subset T^* \C P^1$, to a
neighborhood of a singular point in the K\"ahler-Einstein orbifold. 
We know that under the K\"ahler-Ricci flow on the glued manifold, the area of the $\C P^1$ subvariety is
constant in time.  On the other hand,
under the K\"ahler-Ricci flow, the metric on the orbifold 
increases linearly in time.  This motivates an initial approximate flow on $T^* \C P^1$
given by the potential 
\begin{align} \label{3.1}
\phi_{EH}^{(0)}(t,z) = & 2(t \log t - t) + b^{-1} \phi_{EH, bt} \left( |z|^2 \right) \\
= & 2(t \log t - t) + \frac{1}{b} \sqrt{1 + b^2 t^2 |z|^4} + 
\frac{1}{2b}
\log 
\frac{
\sqrt{1 + b^2 t^2 |z|^4} - 1}{
\sqrt{1 + b^2 t^2 |z|^4} + 1}
 \notag
\end{align}
in a deleted neighborhood of the vertex in $\C^2/\Z_2$.
Here $b$ is a positive constant that determines the area of the $\C P^1$ subvariety; the area is
$2 \pi b^{-1}$.

The corresponding metric has Ricci flat time slices which are asymptotically flat, with a metric at infinity that is linearly increasing in $t$.
The area of the zero section $\C P^1 \subset T^* \C P^1$ is constant in $t$. The curvature of the metric is concentrated
in a region where $|z| \le \const t^{-\frac12}$. 

The $2(t \log t - t)$ term in (\ref{3.1}) is arranged so that the potential flow equation (\ref{2.3}) is
satisfied to leading order.
One finds that 
\begin{equation} \label{3.2}
\frac{\partial \phi_{EH}^{(0)}}{\partial t} - \log \det \left(  \partial_i \overline{\partial}_j \phi_{EH}^{(0)}
\right) = \frac{1}{bt} \sqrt{1 + b^2 t^2 |z|^4}.
\end{equation}
We will eventually want to consider (\ref{2.5}) when $|z| = O \left( t^{-a} \right)$, with $a \in 
\left( 0, \frac12 \right)$, in which case the
right-hand side of (\ref{3.2}) is $O \left( t^{-2a} \right)$. While this is decreasing in $t$, it is not decreasing
fast enough and we need a better approximate solution.

To this end, we first write down what (\ref{2.3}) becomes if we assume a $U(2)$-symmetry.
If $u(z,t) = F(\rho,t)$ with $\rho = |z|^2$ then the K\"ahler form
$\omega = \sqrt{-1} \partial \overline{\partial} u$
is associated to the metric
\begin{equation} \label{3.3}
(g_{i \overline{j}}) = \begin{pmatrix}
F_\rho + |z^1|^2 F_{\rho \rho} & z^1 \overline{z}^2 F_{\rho \rho} \\
z^2 \overline{z}^1 F_{\rho \rho} & F_\rho + |z^2|^2 F_{\rho \rho}
\end{pmatrix}
\end{equation}
and
(\ref{2.3}) becomes
\begin{equation} \label{3.4}
F_t = \log \left( F_\rho (F_\rho + \rho F_{\rho \rho}) \right).
\end{equation}
Next, we do a change of variable to bring the curvature concentration region to unit scale.  That is, we change
variables from
$(t, \rho)$ to $(s, \eta)$, where $s = t$ and $\eta = t \rho$.  After renaming $F$ to $G$, equation (\ref{3.4}) becomes
\begin{equation} \label{3.5}
G_s + \frac{1}{s} \eta G_\eta = \log \left( G_\eta (G_\eta + \eta G_{\eta \eta}) \right) + 2 \log s.
\end{equation}
The approximate solution 
\begin{align} \label{3.6}
G^{(0)}(s, \eta) = & 2 (s \log s - s) + b^{-1} \phi_{EH,b}(\eta) \\
= & 2(s \log s - s) + \frac{1}{b} \sqrt{1 + b^2 \eta^2} + 
\frac{1}{2b}
\log 
\frac{
\sqrt{1 + b^2 \eta^2} - 1}{
\sqrt{1 + b^2 \eta^2} + 1}
 \notag
\end{align}
satisfies
\begin{equation} \label{3.7}
G^{(0)}_s = \log \left( G^{(0)}_\eta (G^{(0)}_\eta + \eta G^{(0)}_{\eta \eta}) \right) + 2 \log s.
\end{equation}

We now write a formal solution 
\begin{equation} \label{3.8}   
G(s, \eta) = G^{(0)}(s, \eta) + \frac{1}{s} G^{(1)}(\eta) +
\frac{1}{s^2} G^{(2)}(\eta) + \ldots
\end{equation}
and substitute it into (\ref{3.5}) in order to find the terms $\{ G^{(j)} \}_{j=1}^\infty$ iteratively
by equating orders of $s$.
One gets equations of the form
\begin{equation} \label{3.9}
G^{(j)}_{\eta \eta} + \left( \frac{1}{\eta} + \frac{b^2 \eta}{1+b^2 \eta^2} \right) G^{(j)}_{\eta} = H^{(j)},
\end{equation}
where $H^{(j)}$ is a function of $\eta$ constructed from $\{G^{(1)}, \ldots, G^{(j-1)}\}$, which appear polynomially.
The relevant solution to (\ref{3.9}) is
\begin{equation} \label{3.10}
G^{(j)}(\eta) = \int_0^\eta \frac{1}{\sigma \sqrt{1+b^2 \sigma^2}} \int_0^\sigma  \tau \sqrt{1+b^2 \tau^2} \:
H^{(j)}(\tau) d\tau d\sigma.
\end{equation}
For example, $H^{(1)}(\eta) = 1$ and
\begin{equation} \label{3.11}
G^{(1)}(\eta) = \frac{1}{3b^2} \left[ \frac12 b^2 \eta^2 +
\log \frac{\sqrt{1+b^2 \eta^2} + 1}{2}
\right].
\end{equation}
\begin{lemma} \label{lem1}
$G^{(j)}(\eta) = \frac{1}{(j+1)3^j} \eta^{j+1} + O(\eta^j)$ as $\eta \rightarrow \infty$,
with similar asymptotics for the derivatives of $G^{(j)}$.
\end{lemma}
\begin{proof}
The large-$\eta$ asymptotics of $G^{(0)}$ are $G^{(0)}_{asymp}(\eta) = 2(s \log s - s) + \eta$.
We can do a similar iteration procedure as in (\ref{3.5})-(\ref{3.10}), starting with $G^{(0)}_{asymp}$, to obtain
$\{ G^{(j)}_{asymp} \}_{j=1}^\infty$. One finds that they are all polynomials in $\eta$.  To say what 
polynomials they are, we can use the fact that an exact solution
of (\ref{3.5}) is given by the expanding complex hyperbolic K\"ahler-Ricci flow
\begin{align} \label{3.12}
G & = 2(t \log t - t) - 3 t \log \left( 1 - \frac{1}{3} \rho \right) \\
& = 
2(s \log s - s)
- 3 s \log \left( 1 - \frac{\eta}{3s} \right) \notag \\ 
& = 2(s \log s - s) + 
\sum_{j=0}^\infty \frac{1}{(j+1)3^j} s^{-j} \eta^{j+1}. \notag
\end{align}
Equating terms with equal powers of $s$ shows that $G^{(j)}_{asymp}(\eta) = \frac{1}{(j+1)3^j} \eta^{j+1}$.
In view of the construction in (\ref{3.10}), inductively the difference between $G^{(j)}_{asymp}$ and $G^{(j)}$ will be of the
lower order $O(\eta^j)$. A similar argument works for the derivatives.
    \end{proof}

Put 
\begin{equation} \label{3.13}
    \widehat{G}^{(k)}(s, \eta) = G^{(0)}(s, \eta) + \sum_{j=1}^k s^{-j} G^{(j)}(\eta)
    \end{equation}
    and in view of (\ref{3.5}), put
\begin{equation} \label{3.14}
\widehat{F}^{(k)} = \widehat{G}^{(k)}_s +
\frac{1}{s} \eta \widehat{G}^{(k)}_\eta - \log \left( \widehat{G}^{(k)}_\eta 
(\widehat{G}^{(k)}_\eta + \eta \widehat{G}^{(k)}_{\eta \eta}) \right) - 2 \log s.
\end{equation}

The proof of the next lemma has input from ChatGPT-5.5.
\begin{lemma} \label{lem2}
Given $a \in \left(0, \frac12\right)$, if $s$ is large then
in the interval $0 < \frac{\eta}{s} \le s^{-2a}$,
the terms $\widehat{G}^{(k)}_\eta$ and $\widehat{G}^{(k)}_\eta + \eta \widehat{G}^{(k)}_{\eta \eta}$
are positive, and
there is a bound $|\widehat{F}^{(k)}| \le \const s^{-(k+1)} \left( 1 + \eta^{k+1} \right)$.
\end{lemma}
\begin{proof}
Put
$
A(\eta)=\sqrt{1+b^2\eta^2}$ and
$
x=\frac{1+\eta}{s}$.
On $0 <  \eta/s\le s^{-2a}$, as $s \rightarrow \infty$ we have
$x\le s^{-1}+s^{-2a} \rightarrow 0$.
All $O(x^m)$-terms below are uniform in this region.

In general, if we have an expansion
$Y=\sum_{j=1}^k s^{-j}Y_j(\eta)$ with
$|Y_j(\eta)|\le \const(j) (1+\eta)^j$
then
$\log(1+Y)
=
\sum_{\ell=1}^k s^{-\ell}L_\ell(\eta)+O(x^{k+1})$.
This is because $Y$ is $O(x)$, so applying Taylor's theorem
to the function $\log(1 + \cdot)$ gives a remainder
$O(x^{k+1})$, and each Taylor monomial of total $s^{-1}$-degree
$\ell$ in the expansion is $O(x^\ell)$.

Now
$
G^{(0)}_\eta=\frac{A}{b\eta}$ and
$
G^{(0)}_\eta+\eta G^{(0)}_{\eta\eta}=\frac{b\eta}{A}$.
Put $P_j=\frac{b\eta}{A}G^{(j)}_\eta$,
$
Q_j=\frac{A}{b\eta}
\left(G^{(j)}_\eta+\eta G^{(j)}_{\eta\eta}\right)$,
$Y_1=\sum_{j=1}^k s^{-j}P_j$ and
$
Y_2=\sum_{j=1}^k s^{-j}Q_j$.
Then
$
\widehat{G}^{(k)}_\eta
=
\frac{A}{b\eta}(1+Y_1)$ and
$
\widehat{G}^{(k)}_\eta+\eta\widehat{G}^{(k)}_{\eta\eta}
=
\frac{b\eta}{A}(1+Y_2)$.
Equation (\ref{3.10}) implies that as $\eta$ goes to zero,
$G^{(j)}_\eta=O(\eta)$ and
$
G^{(j)}_{\eta\eta}=O(1)$,
so
$
G^{(j)}_\eta+\eta G^{(j)}_{\eta\eta}=O(\eta)$.
It follows that $Q_j$
extends continuously to $\eta=0$.

By the derivative estimates of Lemma \ref{lem1}
through order $2$, we have
$
|P_j|+|Q_j|\le \const(j)(1+\eta)^j$.
Hence
$
|Y_1|+|Y_2|\le \const\sum_{j=1}^k x^j$, which is bounded above by
$\const x$ when $s$ is large.
For large $s$, both $1+Y_1$ and $1+Y_2$ are positive. Hence
$\widehat{G}^{(k)}_\eta>0$ and
$\widehat{G}^{(k)}_\eta+\eta\widehat{G}^{(k)}_{\eta\eta}>0$.

The leading determinant factors in $\widehat{F}^{(k)}$ cancel, and
\begin{equation}
\widehat{F}^{(k)}
=
\frac{A}{bs}
+
\sum_{j=1}^k s^{-j-1}
\left(\eta G^{(j)}_\eta-jG^{(j)}\right)
-
\log(1+Y_1)-\log(1+Y_2).
\end{equation}
The log terms have an expansion through order $s^{-k}$, with error $O(x^{k+1})$.
Also, the derivative estimates give
$
\frac{A}{bs}=O(x)$ and
$
s^{-j-1}\left(\eta G^{(j)}_\eta-jG^{(j)}\right)=O(x^{j+1})$. Hence
$\widehat{F}^{(k)}$ has the form
\begin{equation}
\widehat{F}^{(k)}
=
\sum_{\ell=1}^k s^{-\ell}C_\ell(\eta)+O(x^{k+1}).
\end{equation}

Here $C_\ell$ is the formal residue coefficient of
order $s^{-\ell}$, for $1\le \ell\le k$.  By construction, it vanishes.
Hence
$\widehat{F}^{(k)}=O(x^{k+1})$.
Finally,
$
x^{k+1}
=
s^{-(k+1)}(1+\eta)^{k+1}
\le
\const s^{-(k+1)}(1+\eta^{k+1})$,
which proves the estimate.
\end{proof}

We now scale back and put
$\phi_{EH}^{(k)}(t,z) = \widehat{G}^{(k)}(t, t |z|^2)$. 
Explicitly, if $k=1$ then
\begin{align} \label{3.18}
\phi_{EH}^{(1)}(t,z)
= & 2(t \log t - t) + \frac{1}{b} \sqrt{1 + b^2 t^2 |z|^4} + 
\frac{1}{2b}
\log \frac{
\sqrt{1 + b^2 t^2 |z|^4} - 1}{
\sqrt{1 + b^2 t^2 |z|^4} + 1}
+  \\
& \frac{1}{3b^2 t} \left[  \frac12 b^2 t^2 |z|^4 +
\log \frac{\sqrt{1+b^2 t^2 |z|^4} + 1}{2}
\right]. \notag
\end{align}

Lemma \ref{lem2} and (\ref{3.3}) imply that $\phi_{EH}^{(k)}(t,\cdot)$ is strictly plurisubharmonic for large $t$
in the region $|z| \le t^{-a}$.
Defining $\omega_{EH}^{(k)} = \sqrt{-1} \partial \overline{\partial} \phi_{EH}^{(k)}$ and
\begin{equation} \label{3.19}
f_{EH}^{(k)}(t,z) = \frac{\partial \phi_{EH}^{(k)}}{\partial t} - \log \det \left(  \partial_i \overline{\partial}_j \phi_{EH}^{(k)}
\right),
\end{equation}
we have
\begin{equation} \label{3.20}  
\sqrt{-1} \partial \overline{\partial} f_{EH}^{(k)} = \frac{d \omega_{EH}^{(k)}}{dt} + \Ric(\omega_{EH}^{(k)}).
\end{equation}

\begin{lemma} \label{lem3}
Given $a \in \left( 0, \frac12 \right)$, if $t$ is large then in the region $|z| \le t^{-a}$
the magnitudes of $f_{EH}^{(k)}$
 and $t \partial_t f_{EH}^{(k)}$
are bounded above by $\const \left( |z| + t^{- \frac12} \right)^{2(k+1)}$,
and the magnitude of $\nabla f_{EH}^{(k)}$ is bounded above by
$\const t^{- \frac12} \left( |z| + t^{- \frac12} \right)^{2k+1}$.
\end{lemma}
\begin{proof}
The bound on $f_{EH}^{(k)}$ follows from Lemma \ref{lem2}.
In this region, $\omega_{EH}^{(k)}$ is uniformly biLipschitz equivalent to $\omega_{EH}^{(0)}$, so 
to bound $|\nabla f_{EH}^{(k)}|$ it is enough to estimate
the sum $g_{EH}^{(0),\overline{j} i} \partial_{z^i} f_{EH}^{(k)} \partial_{\overline{z}^j} f_{EH}^{(k)}$.
From the chain rule, 
\begin{equation} \label{3.21}
\frac{\partial f_{EH}^{(k)}}{\partial \overline{z}^j} = 
\frac{\partial \eta}{\partial \overline{z}^j} \frac{\partial \widehat{F}^{(k)}}{\partial \eta} =
t z_j \frac{\partial \widehat{F}^{(k)}}{\partial \eta} = s z_j \frac{\partial \widehat{F}^{(k)}}{\partial \eta}.
\end{equation}
From (\ref{2.8}), we have
\begin{equation} \label{3.22}
g_{EH}^{(0),\overline{j} i} z_{\overline{j}} z_i = \frac{\sqrt{1+b^2 \eta^2}}{b s^2},
\end{equation}
so
\begin{equation} \label{3.23}
|\nabla  f_{EH}^{(k)} |^2 = \frac{\sqrt{1+b^2 \eta^2}}{b}
\left( \frac{\partial \widehat{F}^{(k)}}{\partial \eta} \right)^2.
\end{equation}
By the argument in the proof of Lemma \ref{lem2}, in the given region we have
$| \frac{\partial \widehat{F}^{(k)}}{\partial \eta} | \le \const s^{-(k+1)} (1 + \eta^k)$,
from which the bound on $|\nabla  f_{EH}^{(k)} |$ follows.

Next, by the chain rule,
\begin{equation} \label{3.24}
\frac{\partial f_{EH}^{(k)}}{\partial t} = 
\frac{\partial \eta}{\partial t} \frac{\partial \widehat{F}^{(k)}}{\partial \eta} +
\frac{\partial s}{\partial t} \frac{\partial \widehat{F}^{(k)}}{\partial s} =
|z|^2 \frac{\partial \widehat{F}^{(k)}}{\partial \eta} +
\frac{\partial \widehat{F}^{(k)}}{\partial s} = s^{-1} \eta \frac{\partial \widehat{F}^{(k)}}{\partial \eta} +
\frac{\partial \widehat{F}^{(k)}}{\partial s}.
\end{equation}
Estimating $|\frac{\partial \widehat{F}^{(k)}}{\partial s}|$ similarly and
applying the previous estimate on $| \frac{\partial \widehat{F}^{(k)}}{\partial \eta} |$ gives the bound on 
$t |\frac{\partial f_{EH}^{(k)}}{\partial t}|$.
\end{proof}

\section{Model flow} \label{sect4}

In this section we construct the model K\"ahler potential $\phi_{mod}(t)$ on a manifold $M$
as in the statement of Theorem \ref{thm1}. The potential is obtained by gluing the potential
$\phi_{EH}^{(1)}(t)$ from Section \ref{sect3} to a K\"ahler potential $\phi_X(t)$ for the expanding flow on the
orbifold $X$, in a neighborhood of each singular point of $X$. The only ambiguities in the
construction of $\phi_{mod}(t)$ are the choice of a bump function $\sigma$ and the number
$a \in \left( 0, \frac12 \right)$ that determines the scale $|z| \sim t^{-a}$ at which we do the gluing.
To get the best estimates for Theorem \ref{thm1}, we will take $a$ to be $\frac13$.

We will estimate the deviation of $\phi_{mod}(t)$ from satisfying the potential flow equation.
The individual terms $\phi_{EH}^{(1)}(t)$ and $\phi_X(t)$ satisfy good estimates for the deviation.
To leading order, $\phi_{EH}^{(1)}(t)$ and $\phi_X(t)$ agree on the gluing region.
However, the discrepancy between them introduces some errors to the result of gluing, in terms of
satisfying the potential flow equation,
that need to be controlled.  There are two main sources of error.  One source is a lower order term in
$\phi_{EH}^{(1)}(t)$ that doesn't appear in $\phi_X(t)$. The other source is the difference between
the order-$t$ terms in $\phi_{EH}^{(1)}(t)$ and $\phi_X(t)$. Fortunately, both of these error terms
are harmonic to leading order.

We now start the construction.
Let $X$ be a compact K\"ahler orbifold of complex dimension two that admits a K\"ahler-Einstein orbifold metric
$\omega_{KE}$ satisfying $\Ric \left( \omega_{KE}\right) = - \omega_{KE}$. We assume that $X$ has isolated
singular points with isotropy group $\Z_2$. 

\begin{example} \cite[Section 10]{Song-Weinkove (2014)}
Let $Z_1$ and $Z_2$ be hyperelliptic Riemann surfaces of genus at least two, with involutions $i_1$ and $i_2$, respectively. Let $\omega_{Z_i}$ be a constant curvature metric on $Z_i$, normalized so that
$\Ric \left( \omega_{Z_i}\right) = - \omega_{Z_i}$. Let $p_i : Z_1 \times Z_2 \rightarrow Z_i$ be the
projection map.  Then $\omega_{Z_1 \times Z_2} = p_1^* \omega_{Z_1} + p_2^*\omega_{Z_2}$ is a 
K\"ahler-Einstein metric on
$Z_1 \times Z_2$, with $\Ric \left( \omega_{Z_1 \times Z_2} \right) = - \omega_{Z_1 \times Z_2}$. Put
$X = (Z_1 \times Z_2)/\Z_2$, the quotient by the diagonal $\Z_2$-action.  Then $X$ is a
K\"ahler-Einstein orbifold with $\Ric \left( 
\omega_{KE} \right) = - \omega_{KE}$, having isolated singularities
with $\Z_2$-isotropy groups.
\end{example}

There is an orbifold K\"ahler-Ricci flow $\omega_X$ on $X$ with $\omega_X(t) = t \omega_{KE}$. 

Given a singular point
$x \in X$, some neighborhood $U_x$ of $x$ is analytically equivalent to $B(0, \delta)/\Z_2$, for some ball $B(0, \delta) \subset\C^2$.
Letting $p : B(0, \delta) \rightarrow U_x$ be the quotient map, $p^* \omega_{KE}$ is a smooth K\"ahler-Einstein metric
on $B(0, \delta)$. 
Let $\phi_{KE}$ be a potential for $p^* \omega_{KE}$ in Bochner coordinates. 
This means
that $p^* \omega_{KE} = \sqrt{-1} \partial \overline{\partial} \phi_{KE}$
and there is an expansion $\phi_{KE}(z) = |z|^2 + \sum_{|J|,|K| \ge 2} c_{J \overline{K}}
z^J \overline{z}^K$, where $J$ and $K$ are multi-indices; cf.
\cite[p. 50]{Arezzo-Loi (2004)}.  

In terms of the potential $\phi_{KE}$, the K\"ahler-Einstein condition
becomes
$\log \frac{(p^* \omega_{KE})^n}{\Omega} = \phi_{KE} + F$, where $F$ is a real pluriharmonic
function. We can assume that $F = h + \overline{h}$ for a holomorphic function $h$.
To be compatible with the expansion of $\phi_{KE}$ in Bochner coordinates, one sees that
$h$ must be
an imaginary constant, so $F = 0$. 

Writing out the expansion of $\phi_{KE}$, it takes the form
$\phi_{KE}(z) = |z|^2 + 
C_{a \overline{b} c \overline{d}} z^a \overline{z}^b z^c \overline{z}^d + O(|z|^6)$, 
(A priori there is an $O(|z|^5)$ term, but $\phi_{KE}$ is invariant under the $\Z_2$-symmetry
$z \rightarrow -z$.)
Here $C_{a \overline{b} c \overline{d}}$ is proportional to the curvature tensor at $0$ and has the
symmetries $C_{a \overline{b} c \overline{d}} = 
C_{c \overline{b} a \overline{d}} =
C_{a \overline{d} c \overline{b}} =
C_{c \overline{d} a \overline{b}}$.
The K\"ahler-Einstein condition implies that $\sum_{a=1}^2 C_{a \overline{a} c \overline{d}} = \frac14 \delta_{cd}$.

We will implicitly use the same notation when
the K\"ahler potential is descended from $B(0,\delta) \subset \C^2$ to $U_x$. A potential for the orbifold K\"ahler-Ricci flow on $B(0, \delta)$ is
$\phi_X(t) = 2(t \log t - t) + t \phi_{KE}$.

For simplicity, we assume hereafter that $X$ only has one singular point $x$.
It is straightforward to extend to the case of more singular points.
Let $\sigma : [0, \infty) \rightarrow [0,1]$ be a smooth nonincreasing function so that
$\sigma \Big|_{[0, \frac12]} = 1$ and $\sigma \Big|_{[1, \infty)} = 0$.
Given $a \in (0, \frac12)$ and $|z| < \delta$, define the model potential by
\begin{equation} \label{4.1}
\phi_{mod}(t,z) =  \sigma \left( t^{a} |z| \right) \phi_{EH}^{(1)}(t,z) + 
\left( 1 - \sigma \left( t^{a} |z| \right) \right)
\phi_{X}(t,z).
\end{equation}
Put
\begin{equation} \label{4.2}
\omega_{mod}(t,z) =
\sqrt{-1} \partial \overline{\partial} \phi_{mod}(t,z) \text{ if } |z| < \delta.
\end{equation}
We extend it to the rest of $X$ as $\omega_X(t)$. We obtain a model flow 
on $M$, the
complex manifold that is the result of gluing a truncated copy of $T^\star \C P^1$ at the orbifold
point of $X$.

Putting
\begin{equation} \label{4.3}
f_{mod}(t,z) = 
\frac{\partial \phi_{mod}}{\partial t} - \log \det \left(  \partial_i \overline{\partial}_j \phi_{mod}
\right) \text{if } |z| < \delta
\end{equation}
and extending it by zero to $M$,
we have
\begin{equation} \label{4.4}  
\sqrt{-1} \partial \overline{\partial} f_{mod} = \frac{d \omega_{mod}}{dt} + \Ric(\omega_{mod}).
\end{equation}

In the rest of this section, we take $a = \frac13$. For brevity, if a point is in 
$M - U_{x}$ then we set $|z|$ to be $\delta$ at that point.

\begin{lemma} \label{lem4}
Given $\alpha \in (0, \frac12]$,
for large $t$ we have uniform bounds 
\begin{enumerate}
\item $t^{4/3} \left( |z| + t^{- \frac12} \right)^2 |f_{mod}(t,z)| \le \const$, and
\item 
$t^{4/3} \left( r + t^{- \frac12} \right)^2
\left( 1 + t^{\frac12} r \right)^{2 \alpha}
\frac{|f_{mod}(m,t) - f_{mod}(m^\prime, t^\prime)|}{\left( d_t^2(m,m^\prime) + |t-t^\prime| \right)^\alpha}
\le \const$
\end{enumerate}
whenever $(m^\prime, t^\prime)$
satisfies
$|z|, |z^\prime| \in \left[ \frac12 r, r + t^{-\frac12} \right]$ and
$t \le t^\prime \le t + \left( 1 + t^{\frac12} r \right)^2$
for some $r \in (0, 2\delta]$.
\end{lemma}
\begin{proof}
When $|z| \le \frac12 t^{-a}$, the claim of part (1) follows from Lemma \ref{lem3}.
Since $f_{mod}(t,z)$ vanishes when
$|z| \ge t^{-a}$, 
the only region to check for part (1) is 
when $\frac12 t^{-a} < |z| < t^{-a}$.

In this region, we write
\begin{equation} \label{4.5}
\phi_{mod}(t,z) = 
\phi_{EH}^{(1)}(t,z) + 
 \left( 1 - \sigma \left( t^{a} |z| \right) \right)
\left( \phi_X(t,z) - \phi_{EH}^{(1)}(t,z) \right).
\end{equation}
Then
\begin{align} \label{4.6}
f_{mod} = & f_{EH}^{(1)} + \frac{\partial}{\partial t} \left( (1-\sigma) \left( \phi_X(t,z)
- \phi_{EH}^{(1)}(t,z) \right)\right)
- \\
& \Tr \log \left( I + \left( \omega^{(1)}_{EH} \right)^{-1} \sqrt{-1} \partial \overline{\partial}
\left( (1-\sigma) \left( \phi_X(t,z)  - \phi_{EH}^{(1)}(t,z) \right)\right)
\right), \notag
\end{align}
where the argument of $\sigma$ is $t^{a} |z|$.

In this region,
$\omega^{(1)}_{EH}$ is asymptotically $t \sqrt{-1}dz^i \wedge d\overline{z}^i$.
To leading order, $G^{(0)}$ is asymptotic to
\begin{equation} \label{4.7}
2(s \log s - s) + \eta - \frac{1}{2 b^2 \eta} = 2(t \log t - t) + t|z|^2 - \frac{1}{2 b^2 t |z|^2}.
\end{equation}
Also, $G^{(1)}$ is asymptotic to
\begin{equation} \label{4.8}
\frac{1}{3b^2} \left( \frac12 b^2 \eta^2 + \log \eta \right) =
\frac{1}{3b^2} \left( \frac12 b^2 t^2 |z|^4 + \log (t |z|^2) \right).
\end{equation}
Along with (\ref{3.13}),
it follows that
$\phi_{EH}^{(1)}$ is asymptotic to 
\begin{equation} \label{4.9}
2(t \log t - t) + t|z|^2 - \frac{1}{2 b^2 t |z|^2} + \frac16 t |z|^4 +
\frac{1}{3b^2 t}  \log (t |z|^2).
\end{equation}
Hence $\phi_X  - \phi_{EH}^{(1)}$ is asymptotic to
\begin{equation} \label{4.10}
t C_{a \overline{b} c \overline{d}} z^a \overline{z}^b z^c \overline{z}^d - \frac16 t |z|^4
+ \frac{1}{2 b^2 t |z|^2} - \frac{1}{3b^2 t}  \log (t |z|^2).
\end{equation}

Since $a = \frac13$, one can check from (\ref{4.6}) that $|f_{mod}(t,z)|$ is $O \left( t^{- \: \frac23} \right)$
in this region.
Here we use the fact that
$\sum_i \partial_i \overline{\partial}_i 
\left( C_{a \overline{b} c \overline{d}} z^a \overline{z}^b z^c \overline{z}^d - \frac16  |z|^4 \right) = 0$,
since $\omega_{KE}$ is a K\"ahler-Einstein metric. We also use the fact that 
$\sum_i \partial_i \overline{\partial}_i
\frac{1}{|z|^2} = 0$. 

Hence we have the bound 
$t^{4/3} \left( |z| + t^{- \frac12} \right)^2 |f_{mod}(t,z)| \le \const$ in this region. 

A similar argument
applies to part (2), where we
use the inequality
\begin{align} \label{4.11}
& \frac{|f_{mod}(m,t) - f_{mod}(m^\prime, t^\prime)|}{\left( d_t^2(m,m^\prime) + |t-t^\prime| \right)^\alpha} = \\
& |f_{mod}(m,t) - f_{mod}(m^\prime, t^\prime)|^{1-2\alpha} 
\left( \frac{|f_{mod}(m,t) - f_{mod}(m^\prime, t^\prime)|}{\sqrt{d_t^2(m,m^\prime) + |t-t^\prime|}}
\right)^{2\alpha} \le \notag \\
& (2\max(|f_{mod}(m,t)|, |f_{mod}(m^\prime, t^\prime)|))^{1-2\alpha} 
\left( \frac{|f_{mod}(m,t) - f_{mod}(m^\prime, t^\prime)|}{\sqrt{d_t^2(m,m^\prime) + |t-t^\prime|}}
\right)^{2\alpha}, \notag
\end{align}
along with the derivative estimates in Lemma \ref{lem3}.
\end{proof}

\section{Proof of Theorem \ref{thm1}} \label{sect5}

In this section we use the analytic setup of \cite{Brendle-Kapouleas (2017)}
to prove Theorem \ref{thm1}.  We use a fixed point theorem to show the existence of
a K\"ahler-Ricci flow that is $K(t)$-biLipschitz close to the model flow, where
$K(t) = 1 + O \left( t^{- \frac13 + \epsilon} \right)$ as $t \rightarrow \infty$.
Since many of the details are as in \cite{Brendle-Kapouleas (2017)}, we just
outline the main steps.
A subsequent parabolic Schauder estimate improves this 
to $K(t) = 1 + O \left( t^{- \: \frac23 +\epsilon} \right)$ as stated in
Theorem \ref{thm1}.

We first define certain weighted H\"older norms,
where the weighting is crucial for the proof.

\begin{definition} \label{defn1}
Given $\alpha \in (0,1)$ and $\gamma, \sigma, \Lambda > 0$, let 
$\parallel u \parallel_{X^{0, \alpha}_{\gamma,\sigma,\Lambda}}$ be the supremum of
\begin{equation} \label{5.1}
t^\gamma \left( r + t^{- \frac12} \right)^\sigma 
\left[
|u(m,t)| + 
\left( 1 + t^{\frac12} r \right)^{2 \alpha}
\frac{|u(m,t) - u(m^\prime, t^\prime)|}{\left( d_t^2(m,m^\prime) + |t-t^\prime| \right)^\alpha} \right],
\end{equation}
where
\begin{itemize}
\item $r \in (0, 2\delta]$, 
\item $\Lambda \le t \le t^\prime \le t + \left( 1 + t^{\frac12} r \right)^2$ and
\item 
$|z|, |z^\prime| \in \left[ \frac12 r, r + t^{-\frac12} \right]$.
\end{itemize}
Here $|z|$ denotes the magnitude of the $z$-coordinate of $m$ if $m \in U_{x}$ and
is $\delta$ otherwise, and similarly for $z^\prime$. The distance $d_t(m, m^\prime)$ is 
measured with $g_{mod}(t)$.
Put
\begin{equation} \label{5.2}
\parallel u \parallel_{X^{1, \alpha}_{\gamma,\sigma,\Lambda}} = 
\parallel u \parallel_{X^{0, \alpha}_{\gamma,\sigma,\Lambda}} +
\parallel \nabla u \parallel_{X^{0, \alpha}_{\gamma,\sigma+1,\Lambda}} +
\parallel \Hess u \parallel_{X^{0, \alpha}_{\gamma,\sigma+2,\Lambda}} + 
\parallel \partial_t u \parallel_{X^{0, \alpha}_{\gamma,\sigma+2,\Lambda}},
\end{equation}
where pointwise norms are taken with respect to $g_{mod}(t)$ and
we use time-$t$ parallel transport along minimizing geodesics to define tensor differences
for the H\"older norm.
\end{definition}

The conditions $|z|, |z^\prime| \in \left[ \frac12 r, r + t^{-\frac12} \right]$ and
$t \le t^\prime \le t + \left( 1 + t^{\frac12} r \right)^2$ mean that $(m^\prime, t^\prime)$ lies
in a quasiparabolic neighborhood of $(m,t)$.
The term $1 + t^{\frac12} r$ appears because 
$1 + t^{\frac12} |z|$
is comparable to the time-$t$
injectivity radius at $m$ for $g_{mod}$; see the proof of Lemma \ref{lem5}.

Let $X^{0, \alpha}_{\gamma,\sigma,\Lambda}$ and $X^{1, \alpha}_{\gamma,\sigma,\Lambda}$ be the
corresponding Banach spaces, where we impose the initial condition $u(\Lambda)=0$ on elements
$u \in X^{1, \alpha}_{\gamma,\sigma,\Lambda}$.

Using Lemma \ref{lem4}, for any $\epsilon > 0$ we can find 
$\alpha, \sigma > 0$ small so that putting $\gamma = \frac43 - \epsilon$, we have that $\| f_{mod} \|_{X^{0, \alpha}_{\gamma+\alpha,\sigma+2,\Lambda}}$ is $o(1)$
as $\Lambda \rightarrow \infty$. Here to bump $\sigma$ up to be positive, we use the fact that the support of $f_{mod}$ is in
$\{(z,t) : |z| \le t^{-a} \}$.

We now rewrite (\ref{2.5}) as
\begin{equation} \label{5.3}
\frac{\partial u}{\partial t} - \triangle_{g_{mod}} u = Q(u) - f_{mod},
\end{equation}
where 
\begin{align} \label{5.4}
Q(u) = & \log
\frac{(\omega_{mod} + \sqrt{-1} \partial \overline{\partial} u)^n}{\omega_{mod}^n} - \triangle_{g_{mod}} u \\
= &
\Tr \log \left( I + \omega_{mod}^{-1} \sqrt{-1} \partial \overline{\partial} u \right) - 
\Tr \left( \omega_{mod}^{-1} \sqrt{-1} \partial \overline{\partial} u \right). \notag
\end{align}
Algebraically, there is some $C < \infty$ so that $|Q(u)| \le C |\omega_{mod}^{-1} \sqrt{-1} \partial \overline{\partial} u|^2$. If $\epsilon$, $\alpha$ and $\sigma$ are sufficiently small then
whenever $\parallel u \parallel_{X^{1,\alpha}_{\gamma,\sigma,\Lambda}} \le 1$, there is a
uniform bound showing that
$\parallel Q(u) \parallel_{X^{0,\alpha}_{\gamma+\alpha,\sigma+2,\Lambda}}$ is $o(1)$ as
$\Lambda \rightarrow \infty$.

We now set up a fixed point problem, taking spatially constant functions into account,
as follows.  Given $v \in X^{1,\alpha}_{\gamma,\sigma,\Lambda}$,
we form $Q(v) - f_{mod} \in X^{0,\alpha}_{\gamma+\alpha,\sigma+2,\Lambda}$ and solve
\begin{equation} \label{5.5}
\frac{\partial u}{\partial t} - \triangle_{g_{mod}} u = Q(v) - f_{mod} + \phi
\end{equation}
for $u \in X^{1,\alpha}_{\gamma,\sigma,\Lambda}$ with
$\int_M u(t) \dvol_{g_{mod}(t)} = 0$,
where $\phi$ is a function just of $t$.
As in \cite{Brendle-Kapouleas (2017)}, when adapted to the case of immortal flows,
there is a unique solution and an estimate 
$\| u \|_{X^{1,\alpha}_{\gamma,\sigma,\Lambda}} \le \const 
\| Q(v) - f_{mod} \|_{X^{0,\alpha}_{\gamma+\alpha,\sigma+2,\Lambda}}$.

(To compare with \cite{Brendle-Kapouleas (2017)}, it is convenient to
rewrite (\ref{5.5}) in terms of the normalized model flow $\widehat{g}_{mod}(t) =
t^{-1} g_{mod}(t)$, with bounded diameter. Putting $\widehat{u}(t) = t^{-1} u(t)$,
$\widehat{v}(t) = t^{-1} v(t)$ and 
$\widehat{t} = \log t$, equation (\ref{5.5}) is
equivalent to 
\begin{equation} \label{5.6}
\frac{\partial \widehat{u}}{\partial \widehat{t}} - \triangle_{\widehat{g}_{mod}} \widehat{u}
+ \widehat{u} = 
\widehat{Q}(\widehat{v}) - f_{mod} + \phi.
\end{equation}
The only constraints to worry about come from constant functions.)

As in \cite{Brendle-Kapouleas (2017)}, the map from $v$ to $u$
gives a continuous map $\tau$ from the Banach space $X^{1,\alpha}_{\gamma,\sigma,\Lambda}$ to itself.  
The unit ball in $X^{1,\alpha}_{\gamma,\sigma,\Lambda}$ is a convex compact subset of
the Banach space
$X^{1,\widetilde{\alpha}}_{\widetilde{\gamma},\sigma,\Lambda}$, where
$\widetilde{\alpha}$ is slightly less than $\alpha$ and $\widetilde{\gamma}$ is slightly
less than $\gamma$.  The map $\tau$ extends to a similarly defined map from
$X^{1,\widetilde{\alpha}}_{\widetilde{\gamma},\sigma,\Lambda}$ to itself.
For large $\Lambda$, the above estimates show that the unit ball in 
$X^{1,\alpha}_{\gamma,\sigma,\Lambda}$ is sent by $\tau$
to itself.  The Schauder fixed point theorem implies that there is some $u \in 
X^{1,\alpha}_{\gamma,\sigma,\Lambda}$ so that $\tau(u) = u$. Applying $\partial \overline{\partial}$
to (\ref{5.5}) with $v=u$ shows that
$\omega_{mod} + \sqrt{-1} \partial \overline{\partial} u$ is a K\"ahler-Ricci flow solution
on $M$ that exists for all time greater than or equal to $\Lambda$. It equals $\omega_{mod}(\Lambda)$
at time $\Lambda$, as $u(\Lambda) = 0$.

Since $u \in 
X^{1,\alpha}_{\gamma,\sigma,\Lambda}$, 
we have
\begin{equation} \label{5.7}
    | \Hess u | \le \frac{\const}{t^\gamma (|z| + t^{- \frac12})^{\sigma + 2}}
    \le \frac{\const}{t^{\gamma - \frac{\sigma}{2} - 1}},
\end{equation}
which is almost $O \left( t^{- \: \frac13} \right)$.

We can improve the convergence rate by using the fact that
\begin{equation} \label{5.8}
    | u | \le \frac{\const}{t^\gamma (|z| + t^{- \frac12})^{\sigma}}
    \le \frac{\const}{t^{\gamma - \frac{\sigma}{2}}} \le \const t^{-1}.
\end{equation}
Denote the norm in $X^{0,\alpha}_{0,0,\Lambda}$ by the $C^\alpha$-norm.
From the estimates in Lemma \ref{lem3} and 
the proof of Lemma \ref{lem4}, when $|z| \le \frac12 t^{-a}$ we have
$\|f_{mod}(t)\|_{C^\alpha} \le \const \max_{|z| \le \frac12 t^{-a}} \left( |z| + t^{-\frac12} \right)^4
\le \const t^{- \: \frac43}$ for large $t$, while if $\frac12 t^{-a} \le |z| \le t^{-a}$ then
$\|f_{mod}(t)\|_{C^\alpha} \le \const t^{4a-2} = \const t^{-\: \frac23}$. Hence
$\|f_{mod}(t)\|_{C^\alpha}$ is $O \left( t^{- \: \frac23} \right)$ on $M$.
Also, since $u \in X^{1,\alpha}_{\gamma,\sigma,\Lambda}$, for any $\epsilon^\prime > 0$ we can choose $\epsilon$, $\alpha$ and $\sigma$ so that
$\|Q(u)(t)\|_{C^\alpha}$ is $O \left( t^{- \: \frac23 +\epsilon^\prime} \right)$ on $M$.
From (\ref{5.5}), we have
\begin{align} \label{5.9}
0 = \frac{d}{dt} \int_M u(t) \dvol_{g_{mod}(t)} = 
& \int_M \left( {Q}(u) - f_{mod} \right) \dvol_{g_{mod}(t)}  + 
\phi(t) \vol_{g_{mod}(t)} + \\ 
& \int_M u(t) \frac{d}{dt} \dvol_{g_{mod}(t)}. \notag
\end{align}
Since $|\frac{d}{dt} \dvol_{g_{mod}(t)}| \le \const \dvol_{g_{mod}(t)}$, we conclude that
$\phi(t)$ is $O \left( t^{- \: \frac23 +\epsilon^\prime} \right)$.
Using Lemma \ref{lem5} below, we can apply the parabolic Schauder lemma \cite[Theorem 4.9]{Lieberman (1996)} to (\ref{5.5}).
It gives a local result of the form 
$\|u\|_{C^{2,\alpha}} \le \const \left( \|u\|_{C^0} + \|Q(u) - f_{mod} + \phi \|_{C^\alpha}\right)$,
where local is in the sense of Lemma \ref{lem5} and the $C^{2, \alpha}$-norm 
is with respect to the uniformly controlled product charts from Lemma \ref{lem5}, or equivalently with respect to 
$g_{mod}(t)$.
Hence $\|u\|_{C^{2,\alpha}}$ is $O \left( t^{- \: \frac23 +\epsilon^\prime} \right)$, so
$\omega_{mod} + \sqrt{-1} \partial \overline{\partial} u$ is $K$-biLipschitz to
$\omega_{mod}$ where $K = 1 + O(t^{- \: \frac23 +\epsilon^\prime})$.

\begin{lemma} \label{lem5}
There are $\rho > 0$ and $C < \infty$ so that for all $(m,t) \in M \times [T, \infty)$, the region
$B(m,\rho) \times [t, t+\rho^2]$, endowed with the Riemannian metric
$g_{mod}(s) + ds^2$, is $C$-close in the $C^3$-topology to a Euclidean product region.
\end{lemma}
\begin{proof}
The statement can be checked for $|z| \ge \frac12 t^{-a}$ using the asymptotic estimates
in the proof of Lemma \ref{lem4}.  For $|z| \le \frac12 t^{-a}$,
using (\ref{3.13}) we will have
uniform $C^3$-closeness between
$g_{mod}^{(1)}(s) + ds^2$ and $g_{mod}^{(0)}(s) + ds^2$. Hence it is enough to verify the claim for 
$g_{mod}^{(0)}(s) + ds^2$.  

Let $\alpha_\tau$ be the automorphism of $T^* \C P^1$ which fixes the exceptional $\C P^1$ and acts on
$\C^2/\Z_2$ by sending $z$ to $\tau z$. 
Then $g_{mod}^{(0)}(t)$ is isometric to $b^{-1} \alpha_{\sqrt{bt}}^* g_{EH}$. From the bounded
geometry of $g_{EH}$, and its conical structure at infinity, we can choose $\rho > 0$
and $C^\prime < \infty$ so that
for any $m \in T^* \C P^1$, the product region $B(m,\rho) \times [0, \rho^2]$ with the Riemannian metric
$b^{-1} g_{EH} + dt^2$ is $C^\prime$-close in the $C^3$-topology
to a Euclidean product region. 
Given $t^\prime \in [T, \infty)$, the same will be true for
$B(m,\rho) \times [t^\prime, t^\prime+\rho^2]$, endowed with the Riemannian metric
$b^{-1} \alpha_{\sqrt{bt^\prime}}^* g_{EH} + dt^2$. Since
$b^{-1} \alpha_{\sqrt{bt}}^* g_{EH} = b^{-1} \alpha_{\sqrt{\frac{t}{t^\prime}}}^* \alpha_{\sqrt{bt^\prime}}^* g_{EH}$ and $\frac{t}{t^\prime} \in \left[ 1, 1+ \frac{\rho^2}{t^\prime} \right]$,
there is some $C < \infty$ independent of $t^\prime$ so that the region 
$B(m,\rho) \times [t^\prime, t^\prime+\rho^2]$, with the metric
$g^{(0)}_{mod}(t) + dt^2$, is $C$-close in the $C^3$-topology to a Euclidean product region.
\end{proof}

\section{Proof of Theorem \ref{thm2}} \label{sect6}

In this section we look at the case when the K\"ahler-Einstein orbifold $X$ is complex hyperbolic.
By doing the gluing at a scale $|z| \sim t^{-a}$ with $a$ close to zero, we show that we can improve
the biLipschitz closeness to $K(t) = 1 + O \left( t^{-2+\epsilon} \right)$. We then prove a
stability result saying that for the potential flow (\ref{2.5}), the $C^0$-norm of the time-$t$ potential
remains close to the $C^0$-norm of an initial time-$T$ potential, if $T$ is sufficiently large.

Suppose that $\omega_{KE}$ is a complex hyperbolic orbifold metric, again with isolated singularities
and $\Z_2$ isotropy group; such orbifolds exist \cite{Reid-Stover (2025)}.
Given $k \ge 0$, $a \in \left( 0, \frac12 \right)$ and
$|z| < \delta$, for large $t$ define 
\begin{equation} \label{6.1}
\phi_{mod}^{(k)}(t,z) = \sigma \left( t^a |z| \right) \phi_{EH}^{(k)}(t,z) + \left( 1 - \sigma \left( t^a |z| \right) \right)
\phi_{X}(t,z)
\end{equation}
and
\begin{equation} \label{6.2}
\omega_{mod}^{(k)}(t,z) =
\sqrt{-1} \partial \overline{\partial} \phi_{mod}^{(k)}(t,z).
\end{equation}
We extend it to the rest of $X$ as $\omega_X(t)$. We obtain a model flow 
on $M$, the
complex manifold that is the result of gluing a truncated copy of $T^\star \C P^1$ to the orbifold
point of $X$.

Putting
\begin{equation} \label{6.3}
f_{mod}^{(k)}(t,z) = 
\frac{\partial \phi_{mod}^{(k)}}{\partial t} - \log \det \left(  \partial_i \overline{\partial}_j \phi_{mod}^{(k)}
\right) \text{if } |z| < t^{-a}
\end{equation}
and extending it by zero to $M$,
we have
\begin{equation} \label{6.4}  
\sqrt{-1} \partial \overline{\partial} f_{mod}^{(k)} = \frac{d \omega_{mod}^{(k)}}{dt} + \Ric(\omega_{mod}^{(k)}).
\end{equation}

\begin{lemma} \label{lem6}
For any $\epsilon > 0$ and $\alpha \in (0,\frac12]$, we can choose $k$ and $a$ so that for large $t$, there
are uniform bounds
\begin{enumerate}
\item $t^{2-\epsilon} \left( |z| + t^{- \frac12} \right)^2 |f_{mod}^{(k)}(t,z)| \le \const$, and
\item 
$t^{2-\epsilon} \left( |z| + t^{- \frac12} \right)^2
\left( 1 + t^{\frac12} |z| \right)^{2 \alpha}
\frac{|f_{mod}^{(k)}(m,t) - f_{mod}^{(k)}(m^\prime, t^\prime)|}{\left( d_t^2(m,m^\prime) + |t-t^\prime| \right)^\alpha}
\le \const$
\end{enumerate}
whenever $(m^\prime, t^\prime)$
satisfies
$|z|, |z^\prime| \in \left[ \frac12 r, r + t^{-\frac12} \right]$ and
$t \le t^\prime \le t + \left( 1 + t^{\frac12} r \right)^2$
for some $r \in (0, 2\delta]$.
\end{lemma}
\begin{proof}
For part (1), it is enough to look at the region $|z| < t^{-a}$, as $f_{mod}^{(k)}$ vanishes outside of the region.
The region $|z| \le \frac12 t^{-a}$ is covered by Lemma \ref{lem3}, so we can assume that
$\frac12 t^{-a} < |z| < t^{-a}$.
In this region, we can write
\begin{equation} \label{6.5}
\phi_{mod}^{(k)}(t,z) = 
\phi_{EH}^{(k)}(t,z) + 
\left( 1 - \sigma \left( t^a |z| \right) \right)
\left( \phi_X(t,z) - \phi_{EH}^{(k)}(t,z) \right).
\end{equation}
Then
\begin{align} \label{6.6}
f_{mod}^{(k)} = & f_{EH}^{(k)} + \frac{\partial}{\partial t} \left( (1-\sigma) ( \phi_X - \phi_{EH}^{(k)} )\right)
- \\
& \Tr \log \left( I + \left( \omega^{(k)}_{EH} \right)^{-1} \sqrt{-1} \partial \overline{\partial}
\left( (1-\sigma) ( \phi_X - \phi_{EH}^{(k)} )\right)
\right), \notag
\end{align}
where the argument of $\sigma$ is $t^a |z|$. Lemma \ref{lem3} gives the bounds for $f_{EH}^{(k)}$.

Since $\omega_{KE}$  is complex hyperbolic, we have
\begin{equation} \label{6.7}
\phi_X = 2(t \log t - t) -3 t \log \left( 1 - \frac13 |z|^2 \right).
\end{equation}
The leading asymptotics of $G^{(0)}$ are given in (\ref{4.7}).
In view of Lemma \ref{lem1} and (\ref{6.7}), in the given region, the relevant terms of
$\phi_X - \phi_{EH}^{(k)}$ are
$\frac{1}{2 b^2 t |z|^2} + \const t|z|^{2k+4}$. 
(The omitted logarithmic and lower-order
terms coming from the large-$\eta$ expansions of the $G^{(j)}$ have the same or
better size after applying $\partial_t$ and $t^{-1}\partial\bar\partial$, even
including cutoff derivatives.)
Taking $k$ large, we can neglect the 
$\const t|z|^{2k+4}$ term.
Also,
$\omega^{(k)}_{EH}$ is asymptotically $t \sqrt{-1}dz^i \wedge d\overline{z}^i$.
Using the fact that $\partial_i \overline{\partial}_i \frac{1}{|z|^2} = 0$, one finds that
$f_{mod}^{(k)}$ is $O \left( t^{4a-2} \right)$ in the given region, from which
part (1) follows. The proof of part (2) is similar.
\end{proof}

We can set $a = \frac{1}{k}$  and take $k$ large. The proof of Theorem \ref{thm1} goes through in the complex hyperbolic setting, changing $\gamma$
from $\frac43 - \epsilon$ to $2 - \epsilon$. 

Finally, we give a stability result for the flow.

\begin{proposition} \label{prop1}
There are some $a>0$ and $k \in \Z^+$ with the following property.
Construct the model flow $g_{mod}^{(k)}$ with gluing at scale
$|z| \sim t^{-a}$.
Given $\epsilon > 0$, there is some $T < \infty$ so that if
$u(t)$ is a solution to (\ref{2.5}) on $[T, \infty)$ 
with $\omega_{mod}^{(k)}(t) + \sqrt{-1} \partial \overline{\partial} u(t)$ positive then 
$\| u(t) \|_{C^0} \le \| u(T) \|_{C^0} + \epsilon$ for all $t \ge T$.
\end{proposition}
\begin{proof}
We first claim that for any $\mu > 0$, if $a$ is sufficiently small and $k$ is sufficiently large
then $\|f_{mod}^{(k)}\|_{C^0}$ is $O \left( t^{-2+\mu} \right)$. In the region $\frac12 t^{-a} <
|z| < t^{-a}$, the claim follows from Lemma \ref{lem6}.  Lemma \ref{lem3} implies that the claim is
true in the region $|z| \le \frac12 t^{-a}$ if we take $k$ large enough.

If $\mu$ is small then $\lim_{T \rightarrow \infty} \int_T^\infty t^{-2+\mu} \: dt = 0$. 
Applying the maximum principle to (\ref{2.5}) gives $\frac{d}{dt} \max u(t) \le - \min f_{mod}^{(k)}(t)$, where
the inequality is understood in the sense of forward differences.  The minimum principle gives
$\frac{d}{dt} \min u(t) \ge - \max  f_{mod}^{(k)}(t)$. Hence
\begin{equation} \label{6.8}
\max |u(t)| \le \max |u(T)| + \int_T^t \max |f_{mod}^{(k)}(s)| \: ds.
\end{equation}
We can take $T$ large enough that $\int_T^\infty \max |f_{mod}^{(k)}(s)| \: ds
\le \epsilon$.
\end{proof}

\end{document}